\documentclass[12pt,reqno]{amsart}
\usepackage[activeacute,english]{babel}

\usepackage[utf8]{inputenc}
\usepackage{amsmath} %General package for maths
\usepackage{amsthm} %Package for theorems
\usepackage{amssymb} %Symbols (loads also amsfonts)
\usepackage{amscd} %Package for rectangular diagrams

\usepackage{emptypage}
\usepackage{enumitem} %Enumerations
\usepackage{mathtools} %More symbols, etc.
\usepackage{mathrsfs} %Calligraphic symbols with \mathscr
\usepackage{upgreek}
\usepackage{hyperref}
\usepackage{cleveref}
\usepackage{autonum}

\usepackage{esint} %To write averaged integrals
\usepackage{graphicx,caption} %For using \includegraphics
\usepackage{float} %Interface for defining floating objects

\usepackage{pgf,tikz,pgfplots}
\usetikzlibrary{arrows}
\usetikzlibrary{intersections}
\usetikzlibrary{fadings}
\usepackage{wrapfig}
\usepackage{xcolor}
\definecolor{CBgreen}{HTML}{009E73}
\definecolor{CBorange}{HTML}{E69F00}

\usepackage[margin=0.9in]{geometry}
\newtheorem{theorem}{Theorem}[section]
\newtheorem{proposition}[theorem]{Proposition}
\newtheorem{lemma}[theorem]{Lemma}
\newtheorem{corollary}[theorem]{Corollary}

\theoremstyle{definition}

\theoremstyle{remark}
\newtheorem{remark}[theorem]{Remark}
\newtheorem*{remark*}{Remark}

\newcommand{\con}[1]{\mathbb{#1}}
\newcommand{\C}{\con{C}} %Complex
\newcommand{\R}{\con{R}} %Real
\newcommand{\D}{\mathcal{D}}
\newcommand{\RR}{\mathcal{R}}

\newcommand{\Dom}{\mathrm{Dom}}

\newcommand{\disk}{\con{D}} % Disk

\renewcommand{\leq}{\leqslant}

\renewcommand{\geq}{\geqslant}

\numberwithin{equation}{section}
\title[A Payne-Weinberger inequality for quantum dot Dirac operators]{A Payne-Weinberger inequality\\for quantum dot Dirac operators}

\author[J. Duran]{Joaquim Duran}
\address{J. Duran
\newline
Centre de Recerca Matem\`atica, Edifici C, Campus Bellaterra, 08193 Bellaterra, Spain}
\email{jduran@crm.cat}

\author[A. Mas]{Albert Mas}
\address{A. Mas \textsuperscript{1,2}
\newline
\textsuperscript{1}
Departament de Matem\`atiques,
Universitat Polit\`ecnica de Catalunya,
Campus Diagonal Bes\`os, Edifici A (EEBE), Av. Eduard Maristany 16, 08019
Barcelona, Spain
\newline
\textsuperscript{2}
Centre de Recerca Matem\`atica, Edifici C, Campus Bellaterra, 08193 Bellaterra, Spain}
\email{albert.mas.blesa@upc.edu}

\author[T. Sanz-Perela]{Tom\'as Sanz-Perela}
\address{T. Sanz-Perela \textsuperscript{1,2}
\newline
\textsuperscript{1}
Departament de Matem\`atiques i Inform\`atica,
Universitat de Barcelona,
Gran Via de les Corts Catalanes 585, 08007, Barcelona, Spain
\newline
\textsuperscript{2}
Centre de Recerca Matem\`atica, Edifici C, Campus Bellaterra, 08193 Bellaterra, Spain}
\email{tomas.sanz.perela@ub.edu}

\date{\today}

\thanks{The three authors are supported by the Spanish grants PID2025-168310NB-I00 and RED2024-153842-T funded by MICIU/AEI/10.13039/501100011033 and by FSE+. This work is supported by the Spanish State Research Agency, through the Severo Ochoa and Mar\'ia de Maeztu Program for Centers and Units of Excellence in R\&D (CEX2020-001084-M). The first author is also supported by CEX2020-001084-M-20-1 and acknowledges CERCA Programme/Generalitat de Catalunya for institutional support. The third author is also supported by the Spanish grant PID2024-156429NB-I00. \vspace{3pt}}

\begin{document}

\begin{abstract}
    In this work we prove a Payne-Weinberger type inequality for quantum dot Dirac operators defined on bounded and simply connected planar domains. 
    This is a sharp upper bound for their first positive eigenvalue  depending only on the isoperimetric deficit of the domain. To this end, we use a recently studied connection with the so-called $\overline\partial$-Robin Laplacian and we establish an analogous inequality for its first eigenvalue, relying on the corresponding inequality for the Robin Laplacian.
\end{abstract}

\maketitle 

\setcounter{tocdepth}{2}
\makeatletter
\def\l@subsection{\@tocline{2}{0pt}{2.5pc}{5pc}{}}
\makeatother
\tableofcontents

\vspace{-11pt}

\section{Introduction}

The Payne-Weinberger inequality asserts\footnote{
	This was obtained in \cite[Equation~(2.18)]{Payne1961}. 
	Unfortunately, there is a typo in the cited formula: a factor $\pi^2$ should be $\pi$; see  \cite{Antunes2006,Ftouhi2021}.}  
that given $\Omega\subset \R^2$ a bounded 
simply connected (and sufficiently regular) domain, the first eigenvalue $\Lambda_\Omega$ of the Dirichlet Laplacian in $\Omega$ satisfies
\begin{equation} \label{Eq:PW_Dirichlet}
	|\Omega|\Lambda_\Omega - \pi \Lambda_\disk \leq C_{\mathrm{PW}} \left(\dfrac{|\partial\Omega|^2}{4\pi|\Omega|} -1 \right), \quad \text{where} \quad C_{\mathrm{PW}} := \pi j_{0,1}^2\left( \dfrac{1}{J_1^2(j_{0,1})}-1 \right).
\end{equation}
Here, $\disk$ denotes the disk of unit radius centered at the origin, $J_0$ and $J_1$ denote the Bessel functions of the first kind of order $0$ and $1$, respectively, and $j_{0,1}$ is the first positive zero of~$J_0$; recall that $\Lambda_\disk = j_{0,1}^2$. Moreover, equality holds in \eqref{Eq:PW_Dirichlet} if and only if $\Omega$ is a disk.

This inequality can be nicely visualized in terms of the so-called \textit{Blaschke-Santal\'o diagrams}: 
in this context, given $\mathcal{C}$ a class of open sets of $\R^2$, such diagrams are a representation of the planar region $\mathcal{D}_\mathcal{C}$ consisting of all the points $(x,y) \in \R^2$ such that there exists $\Omega\in \mathcal{C}$ for which 
\begin{equation}
	\frac{|\partial \Omega|}{\sqrt{4\pi|\Omega|}} = x
	\quad\text{and}\quad
	|\Omega|\Lambda_\Omega = y,
\end{equation}
where $\Lambda_\Omega$ is defined for a general open set $\Omega\subset\R^2$ using the usual variational formulation.
The region $\mathcal{D}_\mathcal{C}$ allows visualizing which are the possible (scale-invariant) inequalities involving the first eigenvalue of the Dirichlet Laplacian, the perimeter, and the area, for different classes of sets.
In particular, when $\mathcal{C}$ is the class of bounded (and sufficiently regular) open sets, the isoperimetric inequality states that $\mathcal{D}_\mathcal{C}$ lies at the right of the line $ x = 1$, while the Faber-Krahn inequality asserts that $\mathcal{D}_\mathcal{C}$ lies  above the line $y = \pi \Lambda_\disk$.
For the smaller class $\mathcal{C}$ of bounded and
simply connected smooth domains, the Payne-Weinberger inequality \eqref{Eq:PW_Dirichlet} states that $\mathcal{D}_\mathcal{C}$ lies below the parabola 
\begin{equation}
    y = \pi \Lambda_\disk +  C_{\mathrm{PW}} (x^2- 1).
\end{equation}
While the equality in \eqref{Eq:PW_Dirichlet} is attained when $\Omega$ is a disk ---the point $(1, \pi \Lambda_\disk)$ belongs to both the parabola and $\mathcal{D}_\mathcal{C}$---, it is still an open problem (see~\cite[Conjecture 6.1]{Antunes2006}) to improve the right-hand side of \eqref{Eq:PW_Dirichlet} ---that is, to find the optimal curve bounding from above $\mathcal{D}_\mathcal{C}$.
Let us also mention that other classes $\mathcal{C}$ of open sets (such as convex domains, star-shaped domains, polygons, etc.) have been considered in the literature. We refer the reader to the works~\cite{Antunes2006} and~\cite[Sections~3.1 and~4.3]{Ftouhi2021} which contain, among others, nice surveys on known inequalities, numerical simulations, and conjectures related to Blaschke-Santal\'o diagrams for the first eigenvalue of the Dirichlet Laplacian.

\subsection{The setting and main result}

Our goal in the present work is to prove a Payne-Weinberger type inequality (in the sequel, PW-type inequality) for \textit{quantum dot Dirac operators}.
Given $m\geq 0$ (a parameter that typically denotes the mass) and $\theta\in(-\frac \pi 2,\frac {\pi}{2})$, the quantum dot Dirac operator $\D_{\theta,m}$ is the self-adjoint operator in $L^2(\Omega; \C^2)$ defined by
\begin{equation} \label{Eq:Dirac_op_theta}
	\begin{split}
		\mathrm{Dom}(\D_{\theta,m}) &:= \big\{ 
		(\begin{smallmatrix}u\\v\end{smallmatrix}) \in H^1(\Omega;\C^2): \, \cos \theta \, v = i (1-\sin \theta ) (\nu_1+i\nu_2) u  \,\text{ in } H^{1/2}(\partial \Omega;\C) \big\},\vspace{2cm} \\\vspace{2cm}
		\D_{\theta,m}\begin{pmatrix}
		u\\v
	\end{pmatrix} &:= \begin{pmatrix}
			m&-i(\partial_1 - i \partial_2)\\-i(\partial_1 + i \partial_2)&-m
		\end{pmatrix}\begin{pmatrix}
		u\\v
	\end{pmatrix}
	 \quad\text{for all } (\begin{smallmatrix}u\\v\end{smallmatrix})\in\mathrm{Dom}(\D_{\theta,m}).
	\end{split}
\end{equation}
Here and throughout the paper, $\Omega\subset \R^2$ is a bounded domain with $C^2$~boundary, $\nu:=(\nu_1,\nu_2)$ is the unit outward normal vector field on $\partial\Omega$, and $\nabla:=(\partial_1,\partial_2)$ denotes the gradient in $\R^2$. In the mathematical physics literature, quantum dot Dirac operators are used to model electrons conducting electricity in graphene quantum dots and nano-ribbons; we refer to~\cite{AkhmerovBeenakker,Benguria2017Self,Benguria2017Spectral,DuranMasSanzPerela2026,WurmRycerzetAl} and the references therein for more details about this, as well as their basic mathematical properties.

The spectrum of 
$\D_{\theta,m}$, denoted by $\sigma(\D_{\theta,m})$, consists of real eigenvalues of finite multiplicity accumulating only at $\pm \infty$. Moreover, by charge conjugation (see \cite[Appendix A.2]{DuranMasSanzPerela2026}), we know that 
$\lambda\in\sigma(\D_{\theta,m})$ if and only if
$-\lambda\in\sigma(\D_{-\theta,m})$. 
Therefore, to analyze $\sigma(\D_{\theta,m})$ as 
$\theta$ ranges over~$(-\frac \pi 2,\frac {\pi}{2})$, it suffices to investigate $\sigma(\D_{\theta,m})\cap[0,+\infty)$. In particular, this motivates the study of the first (smallest) nonnegative eigenvalue 
\begin{equation}
    \lambda_\Omega(\theta,m) := \min\big(\sigma(\D_{\theta,m})\cap[0,+\infty)\big),
\end{equation}
which actually satisfies $\lambda_\Omega(\theta,m) > m \geq 0$; see \cite[Lemma~3.1]{DuranMasSanzPerela2026}. 
A natural problem is, as for the Dirichlet Laplacian, to find inequalities relating this eigenvalue with geometric quantities associated to $\Omega$, such as its area and perimeter.

A first numerical approach to the previous problem was carried out by Antunes, Benguria, Lotoreichik, and Ourmi\`eres-Bonafos in \cite[Section~8]{Antunes2021} in the case $\theta=m=0$.\footnote{In the literature, the choice $m=0$ is usually referred to as {\em massless Dirac operator} and $\theta=0$ as \textit{infinite mass boundary conditions}.} More precisely,
in~\cite[Figure~1]{Antunes2021} the first eigenvalue $\lambda_\Omega(0,0)$ was plotted, as a function of~$|\partial\Omega|$, for 2500 randomly generated simply connected domains with $|\Omega|=\pi$. 
For the benefit of the reader, in~\Cref{Fig:Plot_Antunes} we provide a similar plot in the scale-invariant\footnote{Note that $	\lambda_{t\Omega} (\theta,m) = t^{-1} \lambda_\Omega(\theta,tm)$ for all $t>0$, which easily follows by inspecting  \eqref{Eq:Dirac_op_theta}.} variables 
\begin{equation}
	\label{Eq:ScaleInvariantVariablesDirac}
	x= \frac{|\partial\Omega|}{\sqrt{4\pi |\Omega|}}
	\quad\text{and}\quad 
	y = \sqrt{|\Omega|}\lambda_\Omega(0,0),
\end{equation}
using the data of 2437 randomly generated simply connected domains with $|\Omega|=\pi$ ---that is, we plot the points $(|\partial\Omega|/(2\pi),\sqrt{\pi} \lambda_\Omega(0,0))\in\R^2$. 
This data was obtained by Pedro R.\! S.\! Antunes by running the same code used to obtain \cite[Figure~1]{Antunes2021}; we are very grateful to Benguria, Lotoreichik, Ourmi\`eres-Bonafos, and, specially, to Antunes for sharing the data with us. 

\begin{figure}[h!]
	\centering
	\includegraphics[width=0.55\linewidth]{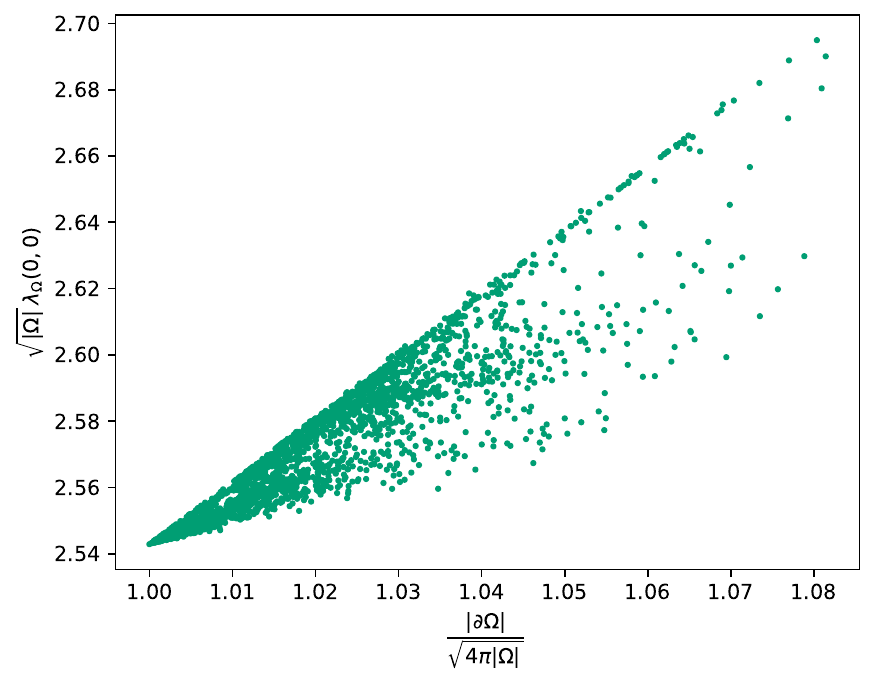}
	\caption{In the scale-invariant variables \eqref{Eq:ScaleInvariantVariablesDirac}, plot of the eigenvalue $\lambda_\Omega(0,0)$, as a function of the perimeter $|\partial\Omega|$, for 2437 randomly generated simply connected domains with $|\Omega|=\pi$. }
	\label{Fig:Plot_Antunes}
\end{figure}

The plot in \Cref{Fig:Plot_Antunes} is a numerical simulation of a Blaschke-Santal\'o diagram for the quantum dot Dirac operator $\D_{0,0}$ that suggests two facts.
On the one hand, the validity of a Faber-Krahn type inequality for $\lambda_\Omega(0,0)$ ---since the point located in the most extreme lower left position corresponds to the disk.
On the other hand, the existence of an upper bound for $\lambda_\Omega(0,0)$ in terms of a \emph{linear} dependence on the perimeter, which could be seen as an analogue of the Payne-Weinberger inequality~\eqref{Eq:PW_Dirichlet} in the Dirac setting.

Regarding the Faber-Krahn type inequality, we will only mention that it is considered a challenging open problem \cite[Problem~5.1]{ProblemListShapeOptimization} for which several (positive) partial results, as well as numerical schemes supporting its validity, have been established; see \cite{Antunes2021,Antunes2024,Benguria2017Spectral,Behrndt2024,Behrndt2026,BrietK2022,Duran2027,DuranMasSanzPerela2026}.
In the present paper, however, we will focus on the linear upper bound suggested by \Cref{Fig:Plot_Antunes}.
Written in the scale-invariant variables  \eqref{Eq:ScaleInvariantVariablesDirac}, we expect to find an upper bound of the form 
\begin{equation}
	y \leq \sqrt{\pi}\lambda_\disk(0,0) +  C (x-1).
\end{equation}
This is actually our main result, which applies not only for the case $\theta=m=0$ but for the whole family of quantum dot Dirac operators $\D_{\theta,m}$ with $m\geq 0$ and $\theta\in (-\frac{\pi}{2},\frac{\pi}{2})$.

\begin{theorem}\label{Thm:DPW_QD}
	Given $m\geq 0$ and  $\theta\in (-\frac{\pi}{2},\frac{\pi}{2})$, there exists a constant $C>0$ depending only on $m$ and $\theta$ such that
	\begin{equation}\label{Ineq_:main_thm}
		\sqrt{|\Omega|}\, \lambda_\Omega  \bigg (\theta, \frac{m}{\sqrt{|\Omega|}} \bigg)
		 - \sqrt{\pi}\, \lambda_\disk \bigg (\theta, \dfrac{m}{\sqrt{\pi} }\bigg)
		 \leq C \left( \dfrac{|\partial\Omega|}{\sqrt{4\pi |\Omega|}}-1 \right)
	\end{equation}
for every bounded simply connected domain~$\Omega\subset \R^2$ with $C^2$~boundary.
\end{theorem}

Next we make some remarks on the constant $C$ appearing in the previous statement, the relation of our result with the plot in \Cref{Fig:Plot_Antunes}, and also the connection with the classical Payne-Weinberger inequality~\eqref{Eq:PW_Dirichlet}. 
For a simpler formula representation, here and in the sequel we will denote the isoperimetric deficit of $\Omega\subset\R^2$ by
\begin{equation}
	\label{Eq:IsopDeficit}
	\delta_\Omega := \dfrac{|\partial\Omega|}{\sqrt{4\pi|\Omega|}}-1.
\end{equation}
Note that, when referring to Blaschke-Santal\'o diagrams in the scale invariant variables \eqref{Eq:ScaleInvariantVariablesDirac}, we have $\delta_\Omega = x - 1$.

\begin{remark}\label{Rmk:Constant}
	We can provide an explicit constant $C$ for which  \eqref{Ineq_:main_thm} holds.
	Indeed, as can be seen in our proof of \Cref{Thm:DPW_QD} ---see in particular \eqref{Eq:ConditionCStar1} and \eqref{Eq:ConditionCStar2} below, taking into account~\eqref{Eq:ConclusionScalingPWDirac}--- the smallest constant $C$ for which we can establish the result is given by 
    \begin{equation} \label{eq:Cthetam}
		C_{\theta,m}:= \max \left\{\sqrt{C_\mathrm{PW}}, \, \sqrt{\pi} K_{\theta,\frac{m}{\sqrt{\pi}}}   \right\}, \, \text{ where } \, 
		{\textstyle K_{\theta,m}  :=  \dfrac{\frac{ C_{\mathrm{PW}}}{\pi}  + \frac{j_{0,1}}{2} }{
				\lambda_\disk(\theta,m) - \tfrac{\frac{1-\sin \theta}{\cos \theta} }{1+ \tfrac{\lambda_\disk(\theta,m)+m}{\lambda_\disk(\theta,m)-m}\left(\frac{1-\sin \theta}{\cos \theta}\right)^2 }}}.
	\end{equation}
	Moreover, for this particular constant $C_{\theta,m}$, we will also show at the end of our proof of~\Cref{Thm:DPW_QD} that equality holds in \eqref{Ineq_:main_thm} if and only if~$\Omega$ is a disk.
	
	Note that $ \lambda_\disk (\theta,m )$ can be computed in terms of Bessel functions, since using separation of variables one can verify that $\lambda_\disk(\theta,m)$ is the first positive solution $\lambda$ to the equation
	\begin{equation} \label{Eq:EigenEq_disk}
		(\lambda+m) \dfrac{1-\sin\theta}{\cos\theta} J_0 \big(\sqrt{\lambda^2-m^2}\big) - \sqrt{\lambda^2-m^2} J_1 \big(\sqrt{\lambda^2-m^2}\big) = 0.
	\end{equation}
	In particular, this shows that indeed the constant $C_{\theta,m}$ in \eqref{eq:Cthetam} only depends on $\theta$ and $m$.
\end{remark}

\begin{remark} \label{Rmk:DPW_inf_mass}
	Let us compare \Cref{Thm:DPW_QD} with the plot in \Cref{Fig:Plot_Antunes}, which corresponds to the case $\theta=m=0$ for domains with $|\Omega|=\pi$.
	On the one hand, an inspection of the data used to generate the plot hints a linear bound of the form 
	\begin{equation}\label{Eq:linear_bound_num}
		\sqrt{\pi} \lambda_\Omega(0,0) \leq \sqrt{\pi}  \lambda_\disk(0,0) + C_{\textrm{N}}\, \delta_{\Omega}, \quad \text{ with } {C_{\textrm{N}} \approx 1.965}. 
	\end{equation}
	On the other hand, for the choice of parameters $\theta=m=0$ and for domains with $|\Omega|=\pi$, our PW-type inequality in~\Cref{Thm:DPW_QD} with the explicit constant~$C_{\theta,m}$ in~\eqref{eq:Cthetam} reads as\footnote{In this case $C_{0,0}=\sqrt{\pi }K_{0,0}$ since $K_{0,0} > \sqrt{C_\mathrm{PW}/\pi}$.
		Indeed, since $\lambda_\disk(0,0) \approx 1.435$, one straightforwardly verifies that 
		$ K_{0,0}  = (C_{\mathrm{PW}}/\pi +j_{0,1}/2)/(\lambda_\disk(0,0)-1/2) \approx 18.056 > 15.675 \approx C_\mathrm{PW}/\pi > \sqrt{C_\mathrm{PW}/\pi}$.}
	\begin{equation} \label{Eq:DPW_inf_mass}
		\sqrt{\pi} \lambda_\Omega(0,0) - \sqrt{\pi}\lambda_\disk(0,0) \leq C_{0,0} \, \delta_\Omega, \quad \text{ with }   C_{0,0} = \sqrt{\pi} \frac{\frac{ C_{\mathrm{PW}}}{\pi}  + \frac{j_{0,1}}{2} }{\textstyle \lambda_\disk(0,0)-\frac{1}{2}} \approx 32.004.
	\end{equation}
	This suggests that \Cref{Thm:DPW_QD} could be true with a constant smaller than the one given in~\Cref{Rmk:Constant} and opens the question of which is the optimal (smallest) constant~$C>0$ for which \Cref{Thm:DPW_QD} holds true. Actually, for large enough isoperimetric deficits $\delta_\Omega$ one can improve the right-hand side of~\eqref{Eq:DPW_inf_mass}.  This is shown in \Cref{Sec:alternative}, where we derive another PW-type inequality for $\lambda_\Omega(0,0)$ using a main result of~\cite{Antunes2021}. As we will see, this approach (which only covers the case $\theta=m=0$) leads to an estimate whose right-hand side is nonlinear in~$\delta_\Omega$, in contrast to what the numerical simulation in~\Cref{Fig:Plot_Antunes} hints. Although it is better than \eqref{Eq:DPW_inf_mass} for large isoperimetric deficits, it is much worse than~\eqref{Eq:DPW_inf_mass} for small isoperimetric deficits.
\end{remark}

\begin{remark}
	Since $\lim_{\theta\downarrow-\frac{\pi}{2}}\lambda_\Omega(\theta, m)=\sqrt{\Lambda_\Omega+m^2}$ (see \cite[Proposition 3.8]{DuranMasSanzPerela2026}), \eqref{Ineq_:main_thm} naturally leads to an inequality of the form
	\begin{equation} 
		\sqrt{|\Omega|\Lambda_\Omega}
		-\sqrt{\pi \Lambda_\disk} \leq C \, \delta_\Omega.
	\end{equation}
	That is, a PW-type inequality for the square root of the first eigenvalue of the Dirichlet Laplacian which is linear with respect to the isoperimetric deficit.
	Unfortunately, the constant obtained here (taking into account \Cref{Rmk:Constant}) does not improve any of the known upper bounds for $\Lambda_\Omega$; see~\cite{Freitas2021,Ftouhi2021}. 
	
\end{remark}

Let us conclude this introduction with a brief summary of our arguments to prove \Cref{Thm:DPW_QD}.
The key ingredient is a connection  between~$\lambda_\Omega(\theta,m)$ and the first eigenvalue~$\mu_\Omega(a)$ of the so-called $\overline\partial$-Robin Laplacian with parameter~$a>0$; we recall these objects in~\Cref{Sec:PW_dbarRobin}, and for more details we refer to~\cite{Duran2026,Duran2027,DuranMasSanzPerela2026}.
In particular, in \Cref{Prop:recipe_bounds} ---which is a restatement of~\cite[Theorem~2.2~$(i)$ and $(iii)$]{Duran2027}--- we recall a recipe to transfer upper bounds from~$\mu_\Omega(a)$ to upper bounds for~$\lambda_\Omega(\theta,m)$, and vice versa. Thanks to this, \Cref{Thm:DPW_QD} is ``reduced'' to obtain a PW-type inequality for~$\mu_\Omega(a)$ with a suitable dependence on the parameter $a$ and the isoperimetric deficit $\delta_\Omega$; see~\Cref{Cor:PW_dabrRobin} in~\Cref{Sec:PW_dbarRobin}.
But since the first eigenvalue~$\mu_\Omega(a)$ is smaller or equal than the first eigenvalue of the Robin Laplacian of the same parameter~$a$, and they coincide if $\Omega = \disk$ (see~\cite[Remark~1.4]{Duran2026}), one can ``reduce'' again the argument to establish a PW-type inequality for the Robin Laplacian (still with this suitable dependence in the constants mentioned above).
This last inequality is established in \Cref{Thm:PW_Robin}, in \Cref{Sec:PW_Robin}, following the ideas of~\cite[Sections II and III]{Payne1961}.
We mention that here is the only point where we need to assume that $\Omega$ is simply connected; see the proof of \Cref{Prop:PW_estimate} for more details. 

This work is organized in reverse order to the previous description of the proof: 
In \Cref{Sec:PW_Robin} we study PW-type inequalities for the Robin Laplacian, in \Cref{Sec:PW_dbarRobin} we transfer them to the $\overline\partial$-Robin Laplacian, and finally in \Cref{Sec:Proof_MainThm} we establish our main result, \Cref{Thm:DPW_QD}. But, first of all, in \Cref{Sec:alternative} we give the alternative approach mentioned in \Cref{Rmk:DPW_inf_mass} to derive a nonlinear PW-type inequality for $\lambda_\Omega(0,0)$. 
In addition, to have a general picture of the case $\theta=m=0$, in \Cref{Ss:final_plot} we gather in a plot some of the estimates discussed in this work (proved and conjectured) together with the data of \Cref{Fig:Plot_Antunes}.

\subsection{An alternative approach}\label{Sec:alternative}

The following inequality is proven in~\cite[Theorem~3]{Antunes2021}: if~$\Omega\subset\R^2$ is a $C^\infty$ simply connected domain, then
\begin{equation}\label{Main_ieq:ABLO2021}
	\lambda_\Omega(0,0)\leq\frac{|\partial\Omega|}{\pi \rho^2+|\Omega|}\lambda_\disk(0,0),
\end{equation}
with equality if and only if $\Omega$ is a disk. Here, $\rho$ denotes the inradius of $\Omega$, that is, the radius of the largest open disk included in $\Omega$; obviously, $\pi \rho^2\leq|\Omega|$. By a simple computation, one can naturally rewrite \eqref{Main_ieq:ABLO2021} as the following PW-type inequality:
\begin{equation}\label{Main_ieq:ABLO2021_PW1}
	\sqrt{|\Omega|}\lambda_\Omega(0,0)
	- \sqrt{\pi}\lambda_\disk(0,0)
	\leq
	\sqrt{\pi}\lambda_\disk(0,0)\Bigg(
	\frac{2|\Omega|}{(\pi \rho^2+|\Omega|)}
	\frac{|\partial\Omega|}
	{\sqrt{4\pi|\Omega|}}-1\Bigg).
\end{equation}
Since the right-hand side in \eqref{Main_ieq:ABLO2021_PW1} does not depend only on the isoperimetric deficit $\delta_\Omega$, as it is the case of our obtained inequality \eqref{Eq:DPW_inf_mass}, it may be desirable to estimate the factor containing~$\rho$ solely in terms of $|\Omega|$ and $|\partial\Omega|$. This can be done using the sharp estimate 
\begin{equation}\label{Ineq:Bonnesen}
2\pi \rho\geq|\partial\Omega|-\sqrt{|\partial\Omega|^2-4\pi|\Omega|},
\end{equation}
for which the equality holds if $\Omega$ is a disk
---\eqref{Ineq:Bonnesen} is a particular instance of the so-called Bonnesen inequalities; see \cite[Theorem 2 and Equation (15)]{Osserman1979}.
Using \eqref{Ineq:Bonnesen} in \eqref{Main_ieq:ABLO2021_PW1} leads to
\begin{equation}\label{Main_ieq:ABLO2021_PW2}
\begin{split}
\sqrt{|\Omega|}\lambda_\Omega(0,0)
- \sqrt{\pi} \lambda_\disk(0,0)
&
\leq
\sqrt{\pi}\lambda_\disk(0,0)
\Bigg(\frac{|\partial\Omega|}{\sqrt{4\pi|\Omega|}}-1+
\sqrt{\frac{|\partial\Omega|^2}{4\pi|\Omega|}-1}\Bigg)\\
&=\sqrt{\pi} \lambda_\disk(0,0)\left(\delta_\Omega+\sqrt{\delta_\Omega(\delta_\Omega+2)}\right).
\end{split}
\end{equation}
Note that the right-hand side of \eqref{Main_ieq:ABLO2021_PW2} behaves like 
$2 \sqrt{\pi} \lambda_\disk(0,0)\, \delta_\Omega$ as 
$\delta_\Omega\uparrow+\infty$, that is, it is asymptotically linear in 
$\delta_\Omega$.  Moreover, since $2\sqrt{\pi}  \lambda_\disk(0,0)\approx 5.086 <32.004\approx C_{0,0}$, \eqref{Main_ieq:ABLO2021_PW2} gives a better estimate than \eqref{Eq:DPW_inf_mass} for large isoperimetric deficits. However, the right-hand side of~\eqref{Main_ieq:ABLO2021_PW2} behaves like $\sqrt{2\pi }\lambda_\disk(0,0)\sqrt{\delta_\Omega}$ as 
$\delta_\Omega\downarrow0$. Therefore, for small isoperimetric deficits,~\eqref{Main_ieq:ABLO2021_PW2} is a much worse estimate than~\eqref{Eq:DPW_inf_mass}.

From these considerations, we realize that the approach developed from \eqref{Main_ieq:ABLO2021} to \eqref{Main_ieq:ABLO2021_PW2} is not adequate to derive a PW-type inequality for $\theta=m=0$ with a right-hand side linear in $\delta_\Omega$, as the numerical simulation described in \Cref{Fig:Plot_Antunes} suggests to exists. 
As we have seen in \eqref{Eq:DPW_inf_mass}, which is a particular case of \Cref{Thm:DPW_QD}, such a linear estimate actually holds. 
However, if one can improve the estimate \eqref{Main_ieq:ABLO2021} from \cite[Theorem 3]{Antunes2021} to 
\begin{equation}\label{Main_ieq:ABLO2021v2}
\lambda_\Omega(0,0)\leq\frac{|\partial\Omega|}{2|\Omega|}\lambda_\disk(0,0)
\end{equation}
---recall that $\pi \rho^2\leq|\Omega|$---, arguing as we did to get \eqref{Main_ieq:ABLO2021_PW1} from \eqref{Main_ieq:ABLO2021}, one would get 
\begin{equation}\label{Main_ieq:ABLO2021v3}
\sqrt{|\Omega|} \lambda_\Omega(0,0)
- \sqrt{\pi}  \lambda_\disk(0,0)
\leq \sqrt{\pi}  \lambda_\disk(0,0)\,\delta_\Omega.
\end{equation}
This would be a better estimate than \eqref{Eq:DPW_inf_mass} since $\sqrt{\pi} \lambda_\disk(0,0)\approx 2.5429 <32.004\approx C_{0,0}$.
Moreover, this estimate still agrees with the numerical simulation described in \Cref{Fig:Plot_Antunes}, since the straight line passing through $(1,\sqrt{\pi} \lambda_\disk(0,0))\in\R^2$ with slope $\sqrt{\pi}\lambda_\disk(0,0)$ stays above the plotted data points; see~\Cref{Fig:Plot_DMS-P} (right). Therefore, at this point a natural question is to prove or disprove~\eqref{Main_ieq:ABLO2021v2} for~$C^\infty$ simply connected domains $\Omega\subset\R^2$. This remains as an open problem.

\subsection{Summary of estimates in a plot} \label{Ss:final_plot}

This short section is devoted to giving a general picture on the main estimates discussed so far in the case $\theta=m=0$ and for domains with area $\pi$.
We do it with the help of \Cref{Fig:Plot_DMS-P}.

\begin{figure}[h!]
	\centering
	\includegraphics[width=0.485\linewidth]{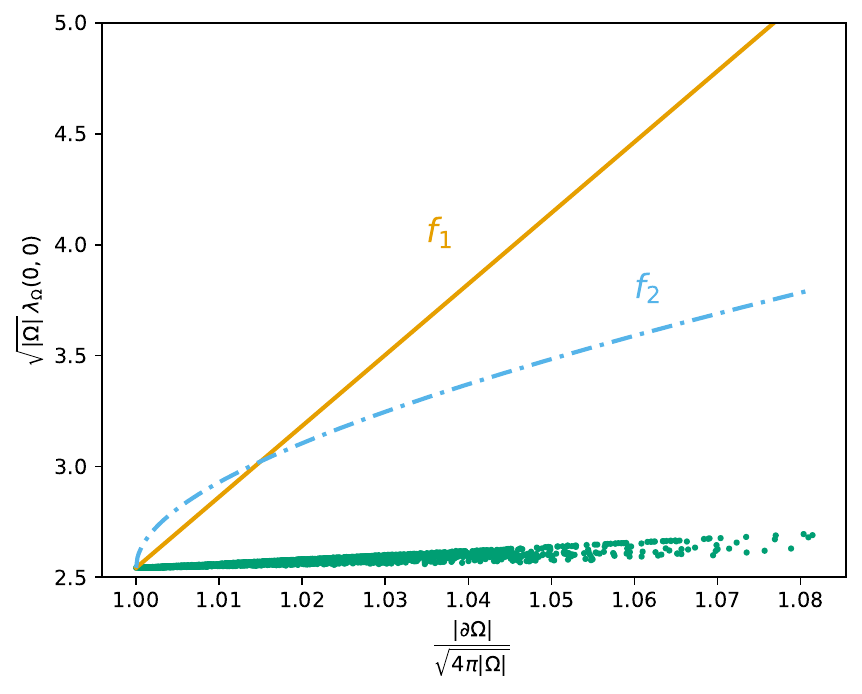}
	\includegraphics[width=0.498\linewidth]{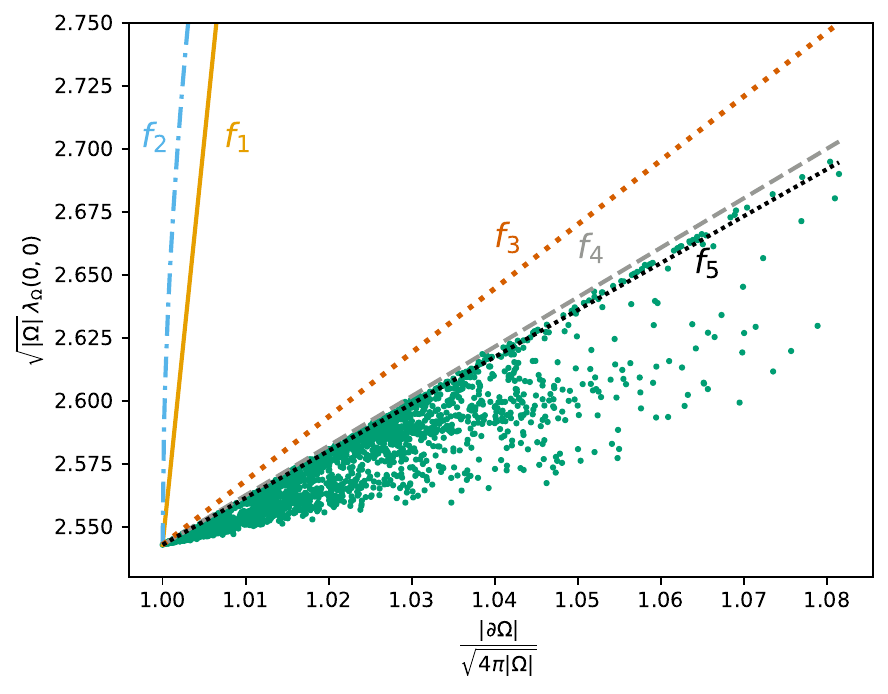}
	\caption{Summary of estimates in the case $\theta=m=0$ and for domains $\Omega$ with area $\pi$. 
	In the left, plot of the numerical data from \Cref{Fig:Plot_Antunes} (green dots) together with the established upper bounds for $\sqrt{\pi} \lambda_\Omega(0,0)$: the linear bound obtained in~\Cref{Thm:DPW_QD} with the explicit constant in~\eqref{eq:Cthetam} (continuous orange line $f_1$) and the nonlinear bound established in~\Cref{Sec:alternative} (dotted-dashed blue line $f_2$).
	In the right, the same picture zoomed with the addition of the conjectured bound~\eqref{Main_ieq:ABLO2021v3} (vermilion dotted line $f_3$), the numerical bound~\eqref{Eq:linear_bound_num} (gray dashed line $f_4$), and the announced lower bound for the optimal constant (black dotted line $f_5$).}
	\label{Fig:Plot_DMS-P}
\end{figure}

The left plot in \Cref{Fig:Plot_DMS-P} gathers the numerical data from \Cref{Fig:Plot_Antunes} and the graph of
\begin{equation}
\begin{split}
&f_1(x):=\sqrt{\pi} \lambda_\disk(0,0) + C_{0,0}( x - 1 ),\\
&f_2(x):=\sqrt{\pi} \lambda_\disk(0,0)+\sqrt{\pi} \lambda_\disk(0,0)
(x-1+\sqrt{x^2-1}),
\end{split}
\end{equation}
for $x\in[1,+\infty)$.
The functions $f_1$ and $f_2$ represent, as functions of 
$x=\delta_{\Omega} +1 $ for domains with $|\Omega|=\pi$, the estimates for $\sqrt{|\Omega|}\lambda_\Omega(0,0)$ given, respectively, by our main theorem ---see~\eqref{Eq:DPW_inf_mass}--- and the discussion in \Cref{Sec:alternative} ---see~\eqref{Main_ieq:ABLO2021_PW2}.
This illustrates what we mentioned in~\Cref{Rmk:DPW_inf_mass} and justifies visually the quest for the optimal constant $C$ for which~\Cref{Thm:DPW_QD} holds true.

The right plot in \Cref{Fig:Plot_DMS-P} is a zoomed version of the left one with three additional functions. First, 
\begin{equation}
	f_3(x):=\sqrt{\pi} \lambda_\disk(0,0) + \sqrt{\pi} \lambda_\disk(0,0)
	(x-1)
\end{equation}
represents the (unproven) inequality \eqref{Main_ieq:ABLO2021v3}; as discussed in \Cref{Sec:alternative}, recall that if one knows that \eqref{Main_ieq:ABLO2021v2} holds true, then \eqref{Main_ieq:ABLO2021v3} shows that $\sqrt{|\Omega|}\lambda_\Omega(0,0)\leq f_3(\delta_{\Omega} +1)$.
Second, for the benefit of the reader, we represent
\begin{equation}
	f_4(x):=\sqrt{\pi } \lambda_\disk(0,0)+C_\mathrm{N}
	\left(x-1\right),
\end{equation}
which is the linear bound \eqref{Eq:linear_bound_num} hinted by the data used to generate the plot in \Cref{Fig:Plot_Antunes}. 
Last, we represent
\begin{equation}
	f_5(x):= \sqrt{\pi } \lambda_\disk(0,0) + C_{\textrm{LB}} (x-1),
\end{equation}
where
\begin{equation}
	C_\mathrm{LB}
	:=  {\textstyle \frac{\sqrt{\pi }\,\lambda_\disk}{6(2\lambda_\disk-1)}
		\left(4(\lambda_\disk-1)-(2\lambda_\disk-3)^2
		\frac{J_2(\lambda_\disk)+J_3(\lambda_\disk)}
		{J_2(\lambda_\disk)-J_3(\lambda_\disk)}-(2\lambda_\disk+1)^2\frac{J_2(\lambda_\disk)-J_1(\lambda_\disk)}
		{J_2(\lambda_\disk)+J_1(\lambda_\disk)}\right)}\approx 1.863.
\end{equation}
Here we used the notation  $\lambda_\disk:=\lambda_\disk(0,0)$, and $J_1$, $J_2$, and $J_3$ denote the Bessel functions of the first kind of order $1$, $2$, and $3$, respectively. The constant $C_\mathrm{LB}$ is a lower bound for the optimal constant. Indeed, in a work in preparation, using perturbative arguments, we will show that if~$C>0$ is such that 
\begin{equation}\label{Eq:PW_Dirac_generalC}
	\sqrt{|\Omega|} \, \lambda_\Omega(0,0) - \sqrt{\pi}\, \lambda_\disk(0,0) \leq C \delta_\Omega
\end{equation}
for every bounded simply connected domain $\Omega\subset \R^2$ with $C^2$~boundary, then $C \geq C_\mathrm{LB}$. Note that 
\begin{equation}
	C_{\textrm{N}} \approx 1.965>1.863\approx C_\mathrm{LB},
\end{equation}
which indicates that $C_\mathrm{LB}$ is a quite tight lower bound for the optimal constant in \eqref{Eq:PW_Dirac_generalC}.

\section{A PW-type inequality for the Robin Laplacian} \label{Sec:PW_Robin}

Consider the eigenvalue problem for the Robin Laplacian
\begin{equation} 
    \begin{cases}
        -\Delta u = \mu u & \text{in } \Omega, \\
        \nu \cdot \nabla u  +a u=0 &\text{on } \partial \Omega,
    \end{cases}
\end{equation}
with boundary parameter $a>0$.
Recall that $\nu$ is the outward unit normal vector on $\partial\Omega$.
It is well known that its first (smallest) eigenvalue $\mu$, denoted by 
$\mu_\Omega^\mathrm{Rob}(a)$, is characterized by 
\begin{equation} \label{Eq:RQ_Robin_mu}
    \mu_\Omega^\mathrm{Rob}(a) = \inf_{u\in H^1(\Omega;\R)\setminus\{0\}}\dfrac{\int_\Omega |\nabla u|^2 + a\int_{\partial\Omega} |u|^2}{\int_\Omega |u|^2};
\end{equation}
see for example \cite[Section 4.1]{Henrot2017}.
The first step in our proof of \Cref{Thm:DPW_QD} is to provide a PW-type inequality for 
$\mu_\Omega^\mathrm{Rob}(a)$; see \Cref{Thm:PW_Robin} below. As far as we know, this inequality (which is of independent interest) is not explicitly stated in the literature, although it readily follows from \cite[Section III]{Payne1961}. 
In its statement, we will use the functions $F:[0,j_{0,1}]  \rightarrow  \R$ and~$G:[0,j_{0,1}]  \rightarrow  \R$ defined by
    \begin{equation}\label{Eq:def_FG}
            F(x):= \pi x^2\left( \dfrac{1}{J_0^2(x)+J_1^2(x)}-1 \right)
        \quad \text{and} \quad
            G(x):= \pi x\dfrac{2J_0(x)J_1(x)}{J_0^2(x)+J_1^2(x)};
    \end{equation}
we recall that $J_0$ and $J_1$ denote the Bessel functions of the first kind of order $0$ and $1$, respectively, and $j_{0,1}$ is the first positive zero of~$J_0$.

\begin{theorem} \label{Thm:PW_Robin}
    Let $\Omega\subset \R^2$ be a bounded simply connected domain with $C^2$~boundary and let~$\delta_\Omega$ denote its isoperimetric deficit, as in \eqref{Eq:IsopDeficit}.
    Then, for all $a>0$ there holds
    \begin{equation} \label{Eq:PW_Robin_a}
        |\Omega| \, \mu_\Omega^\mathrm{Rob} \bigg(\frac{a}{\sqrt{|\Omega|}}  \bigg)
         - \pi \, \mu_\disk^\mathrm{Rob}\bigg(\frac{a}{\sqrt{\pi}}  \bigg) 
        \leq 
        F \left(\sqrt{\mu_\disk^\mathrm{Rob}\big(\tfrac{a}{\sqrt{\pi}}\big)} \right) \delta_\Omega( \delta_\Omega+2) +  G \left(\sqrt{\mu_\disk^\mathrm{Rob}\big(\tfrac{a}{\sqrt{\pi}}\big)} \right) \delta_\Omega,
    \end{equation}
    where $F$ and $G$ are as in \eqref{Eq:def_FG}. In particular, 
    \begin{equation} \label{Eq:PW_Robin_uniform}
        |\Omega| \, \mu_\Omega^\mathrm{Rob} \bigg(\dfrac{a}{\sqrt{|\Omega|}}  \bigg)
        - \pi \, \mu_\disk^\mathrm{Rob}\bigg(\dfrac{a}{\sqrt{\pi}}  \bigg) 
        \leq  
        C_{\mathrm{PW}} \, \delta_\Omega( \delta_\Omega+2) + \pi j_{0,1} \, \delta_\Omega
    \end{equation}
    and the equality holds in \eqref{Eq:PW_Robin_uniform} if and only if $\Omega$ is a disk. Here, $C_{\mathrm{PW}}>0$ is the constant in~\eqref{Eq:PW_Dirichlet}.
\end{theorem}

The estimate \eqref{Eq:PW_Robin_uniform} is a straightforward consequence of \eqref{Eq:PW_Robin_a} and the properties of $F$ and $G$ stated in the following lemma, whose proof is given at the end of this section. 
These properties also show how \eqref{Eq:PW_Robin_a} leads to known upper bounds for the Dirichlet and Neumann Laplacians as $a\uparrow+\infty$ and $a\downarrow0$, respectively; see \Cref{Rmk:Rob_to_DN} below.

\begin{lemma} \label{Lemma:Bessels}
    The functions  $F$ and $G$ defined in \eqref{Eq:def_FG} are continuous in $[0,j_{0,1}]$ and 
    satisfy the following properties:
    \begin{enumerate}[label=$(\roman*)$]
        \item $F$ is strictly increasing in $[0,j_{0,1}]$. In particular, 
        $0=F(0)<F(x)<F(j_{0,1})=C_\mathrm{PW}$ for all $x\in(0,j_{0,1})$, where $C_\mathrm{PW}$ is the constant given in \eqref{Eq:PW_Dirichlet}.
        \item $G(0)=G(j_{0,1})=0<G(x)<\pi j_{0,1}$ for all $x\in(0,j_{0,1})$.
    \end{enumerate}
\end{lemma}

\begin{remark}\label{Rmk:Rob_to_DN}
On the one hand, \eqref{Eq:PW_Robin_a} becomes the classical Payne-Weinberger inequality \eqref{Eq:PW_Dirichlet} for the Dirichlet Laplacian as~$a\uparrow+\infty$. This follows from \Cref{Lemma:Bessels} and the fact that $\lim_{a\uparrow+\infty} \mu_\Omega^\mathrm{Rob}(a) = \Lambda_\Omega$ (see \cite[Proposition 4.5]{Henrot2017}) and $\Lambda_\disk = j_{0,1}^2$. On the other hand, using that~$0\leq\mu_\Omega^\mathrm{Rob}(a)\leq a{|\partial\Omega|}/{|\Omega|}$ for all $a>0$ ---simply take $u=1$ in the variational characterization of $\mu_\Omega^\mathrm{Rob}(a)$--- and \Cref{Lemma:Bessels}, we easily see that both sides of \eqref{Eq:PW_Robin_a} tend to $0$ as~$a\downarrow0$. This later case corresponds to the Neumann Laplacian, where the first eigenvalue is $0$ independently of the domain $\Omega$.
\end{remark}

Upper bounds in the spirit of \eqref{Eq:PW_Robin_a} for the first eigenvalue of the Robin $p$-Laplacian in arbitrary dimensions can be found in \cite[Theorem 1.1]{Amato2024}. However, they are not suitable for our purposes because, first, they are only shown for convex domains and, second, in addition to $\delta_\Omega$, they also depend on the eigenvalue in $\Omega$ (rather than in $\disk$), in contrast to \eqref{Eq:PW_Robin_a}.

A key ingredient in the proof of~\Cref{Thm:PW_Robin} is the following upper bound due to Payne and Weinberger~\cite[Section III]{Payne1961}.\footnote{We point out a typo in the original paper of Payne and Weinberger regarding \eqref{Eq:PW_estimate}: in \cite[Equation (3.3)]{Payne1961} the boundary term $ar_2|v(r_2)|^2$ appears multiplied by a factor $2\pi$. However, in the next paragraph it is asserted that the right-hand side of \eqref{Eq:PW_estimate} is the Rayleigh quotient for the Laplacian in the annulus 
$\{x\in\R^2: r_1<|x|<r_2\}$ with Robin boundary conditions on 
$\{x\in\R^2: |x|=r_2\}$ and Neumann boundary conditions on $\{x\in\R^2: |x|=r_1\}$. Therefore, the factor $2\pi$ in the boundary term cancels out with the one in the solid integrals once written in polar coordinates.} Its proof is not explicitly given there, but only suggested from previous arguments. For the sake of clarity, we give its proof at the end of this section following the ideas of~\cite[Section II]{Payne1961}; a detailed proof is also given in \cite[Section 4]{Freitas2015}.

\begin{proposition}[\cite{Payne1961}] \label{Prop:PW_estimate}
    Let 
    $\Omega\subset \R^2$ be a bounded simply connected domain with $C^2$~boundary. Set
    $$r_1 := \dfrac{1}{2\pi}\sqrt{|\partial\Omega|^2-4\pi |\Omega|} \quad \text{and} \quad r_2:= \dfrac{1}{2\pi}|\partial\Omega|.$$
    Then, 
    \begin{equation} \label{Eq:PW_estimate}
        \mu_\Omega^\mathrm{Rob}(a) \leq  \dfrac{{ \int_{r_1}^{r_2} |\frac{d}{dr} v(r)|^2 \, rdr + ar_2|v(r_2)|^2 } }{{ \int_{r_1}^{r_2} |v(r)|^2 \, rdr } }
    \end{equation} 
for all $a>0$ and all $v\in C^1([r_1,r_2])\setminus\{0\}$.
\end{proposition}

With these two ingredients in hand (\Cref{Prop:PW_estimate} and \Cref{Lemma:Bessels}) we are ready to prove the claimed PW-type inequality for the first eigenvalue of the Robin Laplacian.

\begin{proof}[Proof of \Cref{Thm:PW_Robin}]
Let us address the proof of \eqref{Eq:PW_Robin_a}. We will first show that, without loss of generality, we can assume that $|\Omega|=\pi$. 
To see this, assume that we know that 
\begin{equation}\label{Eq:PW_RobinVolumePi}
    \pi \mu_{\widetilde\Omega}^\mathrm{Rob}(\tilde{a}) - \pi \mu_\disk^\mathrm{Rob}(\tilde{a}) \leq F \left(\sqrt{\mu_\disk^\mathrm{Rob}(\tilde{a})} \right) \delta_{\widetilde \Omega}( \delta_{\widetilde \Omega} + 2) +  G \left(\sqrt{\mu_\disk^\mathrm{Rob}(\tilde{a})} \right) \delta_{\widetilde \Omega}
\end{equation}
for every $\tilde{a}>0$ and every $\widetilde\Omega\subset \R^2$ bounded simply connected domain with $C^2$~boundary such that $|\widetilde \Omega|=\pi$.
Then, by the scale invariance of the isoperimetric deficit and the scaling property
   \begin{equation} \label{Eq:Scale_Robin}
        \mu_{t\Omega}^\mathrm{Rob}(a) = \dfrac{1}{t^2} \mu_\Omega^\mathrm{Rob}(ta) \quad \text{for all } t>0,
    \end{equation} 
taking  $\tilde{a} = \frac{a}{\sqrt{\pi}}$ and $\widetilde \Omega := \sqrt{\frac{\pi}{|\Omega|}} \Omega$ in \eqref{Eq:PW_RobinVolumePi} ---note that $|\widetilde \Omega|=\pi$ and $\delta_{\widetilde \Omega}=\delta_\Omega$---  we deduce that
\begin{equation}
\begin{split}
   |\Omega| \, \mu_\Omega^\mathrm{Rob} \bigg(\dfrac{a}{\sqrt{|\Omega|}}  \bigg)
   - \pi \mu_\disk^\mathrm{Rob}\bigg(\dfrac{a}{\sqrt{\pi}}  \bigg) 
    &= \pi \mu_{\widetilde\Omega}^\mathrm{Rob}\big(\tfrac{a}{\sqrt{\pi}}\big) - \pi \mu_\disk^\mathrm{Rob}\big(\tfrac{a}{\sqrt{\pi}}\big)\\
    &\leq  F\left(\sqrt{\mu_\disk^\mathrm{Rob}\big(\tfrac{a}{\sqrt{\pi}}\big)} \right) \delta_\Omega(\delta_\Omega+2) +  G \left(\sqrt{\mu_\disk^\mathrm{Rob}\big(\tfrac{a}{\sqrt{\pi}}\big)} \right) \delta_\Omega,
\end{split}
\end{equation}
which is \eqref{Eq:PW_Robin_a}. 

We henceforth assume that $|\Omega|=\pi$.
Given $a>0$, let us take in \Cref{Prop:PW_estimate} the competitor 
\begin{equation}
	v(r):= J_0\left( \sqrt{\mu_\disk^\mathrm{Rob}(a) (r^2-r_1^2)} \right)
	\quad\text{for }r\in[r_1,r_2],
\end{equation}
where we recall that
\begin{equation}
	r_1 = \dfrac{1}{2\pi}\sqrt{|\partial\Omega|^2-4\pi |\Omega|} \quad \text{and} \quad r_2= \dfrac{1}{2\pi}|\partial\Omega|.
\end{equation}
To shorten the notation, in what follows we will write $\mu := \mu_\disk^\mathrm{Rob}(a)$.
Using the substitution $t=\sqrt{\mu \, (r^2-r_1^2)}$ and the fact that $r_2^2-r_1^2=1$ (because $|\Omega|=\pi$), by standard properties of the Bessel functions\footnote{Use \cite[Equation (11.3.34)]{Abramowitz1964} to integrate $t J_0(t)^2$, \cite[Equation~(9.1.27)]{Abramowitz1964} to find that $\tfrac{d}{dt} J_0(t) = - J_1(t)$ and $\tfrac{d}{dt}[ \frac{1}{2}t^2 (J_0(t)^2 + J_1(t)^2)) - t J_0(t) J_1(t) ]= t J_1(t)^2 $, and \cite[Equation~(11.3.30) with $\mu=\nu=0$]{Abramowitz1964} to integrate~$t^{-1} J_1(t)^2$.}  we see that
\begin{equation}
    \begin{split}
        &\int_{r_1}^{r_2} |v(r)|^2 \, rdr  = \dfrac{1}{2} \Big( J_0^2(\sqrt{\mu})+J_1^2(\sqrt{\mu}) \Big), \quad
        ar_2|v(r_2)|^2  = a \dfrac{|\partial\Omega|}{2\pi} J_0^2(\sqrt{\mu}), \quad \text{and}
    \end{split}
\end{equation}
\begin{equation}
    \int_{r_1}^{r_2} |\textstyle{\frac{d}{dr} v(r)}|^2 \, rdr = \dfrac{\mu}{2} \Big( J_0^2(\sqrt{\mu})+J_1^2(\sqrt{\mu}) \Big) - \sqrt{\mu} J_0(\sqrt{\mu})J_1(\sqrt{\mu}) + \dfrac{\mu}{2}  \Big(1-\big( J_0^2(\sqrt{\mu})+J_1^2(\sqrt{\mu}) \big) \Big)r_1^2.
\end{equation}
Thus, by \Cref{Prop:PW_estimate},
\begin{equation} \label{Eq:PW_Robin_r1a}
    \mu_\Omega^\mathrm{Rob}(a) \leq \mu + \mu \left(\dfrac{1}{J_0^2(\sqrt{\mu})+J_1^2(\sqrt{\mu})} -1 \right) r_1^2 + 2 \dfrac{a\tfrac{|\partial\Omega|}{2\pi}J_0^2(\sqrt\mu)-\sqrt{\mu}J_0(\sqrt\mu)J_1(\sqrt\mu)}{J_0^2(\sqrt{\mu})+J_1^2(\sqrt{\mu})}.
\end{equation}
Since $\mu = \mu_\disk^\mathrm{Rob}(a)$ is the first eigenvalue of the Robin Laplacian in $\disk$ with boundary parameter~$a>0$, it holds that 
\begin{equation}
    aJ_0(\sqrt\mu) = \sqrt\mu J_1(\sqrt\mu);
\end{equation}
see, for example, \cite[Section 7.2.2]{Henrot2006}. Therefore, we can rewrite the above inequality as
\begin{equation}
    \mu_\Omega^\mathrm{Rob}(a) \leq \mu + \mu \left(\dfrac{1}{J_0^2(\sqrt{\mu})+J_1^2(\sqrt{\mu})} -1 \right) r_1^2 + \sqrt\mu \dfrac{2J_0(\sqrt\mu)J_1(\sqrt\mu)}{J_0^2(\sqrt{\mu})+J_1^2(\sqrt{\mu})} \left( \dfrac{|\partial\Omega|}{2\pi}-1 \right).
\end{equation}
This leads to \eqref{Eq:PW_RobinVolumePi} using the definition of the functions $F$ and $G$ in \eqref{Eq:def_FG} and noticing that $r_1^2=\delta_\Omega(\delta_\Omega+2)$.
As mentioned before, from \eqref{Eq:PW_RobinVolumePi} and by scaling we obtain \eqref{Eq:PW_Robin_a} for any~$\Omega$.

Let us now address the proof of \eqref{Eq:PW_Robin_uniform}. Since $\mu = \mu_\disk^\mathrm{Rob}(a) < \Lambda_\disk = j_{0,1}^2$ for every $a>0$ ---see~\cite[pg.~981]{Freitas2021}---, by bounding the right-hand side of  \eqref{Eq:PW_Robin_a} from above using the estimates of $F$ and $G$ from \Cref{Lemma:Bessels}, we derive \eqref{Eq:PW_Robin_uniform}. Finally, we address the case of equality. On the one hand, we easily see that both sides of \eqref{Eq:PW_Robin_uniform} vanish if $\Omega$ is a disk thanks to the scaling \eqref{Eq:Scale_Robin} and the fact that $\delta_\Omega=0$. On the other hand, if $\Omega$ is not a disk, then $\delta_\Omega>0$ by the isoperimetric inequality. Therefore, \Cref{Lemma:Bessels}~$(ii)$ yields
$G \big(\sqrt{\mu_\disk^\mathrm{Rob}(a)} \big) \delta_\Omega< \pi j_{0,1}\delta_\Omega$ which, combined with \eqref{Eq:PW_Robin_a}, shows that the inequality in \eqref{Eq:PW_Robin_uniform} must be strict.
\end{proof}

We conclude this section proving \Cref{Lemma:Bessels} and  \Cref{Prop:PW_estimate}, the two ingredients needed in the previous proof.
First, we establish the properties of the functions $F$ and $G$ claimed in~\Cref{Lemma:Bessels}.

\begin{proof}[Proof of \Cref{Lemma:Bessels}]
    Both functions are well-defined and continuous in $[0,j_{0,1}]$, because both $J_0(x)$ and $J_1(x)$ are strictly positive for all $x\in(0,j_{0,1}) \subset (0,j_{1,1})$, where $j_{1,1}$ denotes the first positive zero of~$J_1$ ---recall that the zeros of Bessel functions are interlaced~\cite[Equation (9.5.2)]{Abramowitz1964}.
    
    We first address the proof of $(i)$. Using standard properties of the Bessel functions ---see \cite[Equation (9.1.27)]{Abramowitz1964}---, one can straightforwardly verify that 
    \begin{equation}
        \dfrac{d}{dx} \left( \dfrac{1}{J_0^2(x)+J_1^2(x)}-1 \right) =  \dfrac{2J_1^2(x)}{x(J_0^2(x)+J_1^2(x))^2} > 0 \quad \text{for all } x\in(0,j_{0,1}).
    \end{equation}
    Therefore, $F$ is a product of two strictly increasing functions; hence, it is strictly increasing. This yields 
        $0=F(0)<F(x)<F(j_{0,1})=C_\mathrm{PW}$ for all $x\in(0,j_{0,1})$.
    
    We conclude with the proof of $(ii)$. That $G(0)=G(j_{0,1})=0<G(x)$ for all $x\in(0,j_{0,1})$ is clear. The upper bound simply follows from the arithmetic-geometric inequality, which gives~$G(x)\leq \pi x< \pi j_{0,1}$ for all $x\in(0,j_{0,1})$.
\end{proof}

It only remains to prove \Cref{Prop:PW_estimate}. The proof of this result is based on the method of interior parallels, introduced by Makai in 1959 and refined by P\'olya in 1960 and by Payne and Weinberger in \cite{Payne1961}. It consists in using, in the Rayleigh quotient, trial functions whose level lines are parallel to $\partial\Omega$.

\begin{proof}[Proof of \Cref{Prop:PW_estimate}]
We begin the proof introducing some geometric preliminaries.
In what follows, we denote
\begin{equation}
    \mathfrak{d}(x):=\operatorname{dist}(x,\partial\Omega)\quad\text{ for all } x\in\overline\Omega,
\end{equation}
which is well known to be a Lipschitz function which satisfies $|\nabla \mathfrak{d}|=1$ almost everywhere in~$\Omega$. Given $t\geq0$, set
\begin{equation}
\begin{split} 
A(t):=|\{x\in\Omega:\,\mathfrak{d}(x)<t\}|\quad\text{and}\quad
L(t):=|\{x\in\overline\Omega:\,\mathfrak{d}(x)=t\}|;
\end{split}
\end{equation} 
$A(t)$ denotes the area of the ``interior tubular $t$-neighborhood'' of $\partial\Omega$ and $L(t)$ denotes the length of the ``interior parallel at distance $t$'' from $\partial\Omega$.\footnote{We use $|\cdot |$ in the definition of $A(t)$  to denote two-dimensional Lebesgue/Hausdorff measure, while in the definition of $L(t)$ it denotes one-dimensional Hausdorff measure in $\R^2$.} 
Note that by the coarea formula ---since $|\nabla \mathfrak{d}|=1$ almost everywhere in $\Omega$--- we have  
\begin{equation}\label{Eq:Coarea}
	\begin{split} 
		A(t)=\int_{\{x\in\Omega:\,0<\mathfrak{d}(x)<t\}}|\nabla \mathfrak{d}| \, d x
		=\int_0^t |\{x\in\overline\Omega:\,\mathfrak{d}(x)=s\}|\,ds.
	\end{split}
\end{equation}
As a consequence, $L(t)$ is well defined for almost all $t>0$ and it satisfies 
\begin{equation}
	\frac{d}{dt}A(t)=L(t) \quad \text{ for almost all } t>0.
\end{equation}

Next, we define
\begin{equation}
	\label{Eq:DefinitionR(t)}
	r(t):=\frac{1}{2\pi}\sqrt{|\partial\Omega|^2-4\pi A(t)}
	\quad\text{for $t\geq0$.}
\end{equation} 
The function $t\mapsto r(t)$ is well defined and nonincreasing thanks to the isoperimetric inequality (since $A(t)\leq|\Omega|$ for all $t\geq0$) and the monotonicity of $t\mapsto A(t)$, respectively. Moreover,
\begin{equation}\label{Eq:parallel_endpoints}
	\begin{split}
		&r(0)=\frac{1}{2\pi}|\partial\Omega|=:r_2
		\quad\text{and}\\
		&r(t)=\frac{1}{2\pi}\sqrt{|\partial\Omega|^2-4\pi |\Omega|}=:r_1\quad\text{for all $t\geq R:={\textstyle \sup_{x\in\Omega}}\,\mathfrak{d}(x)$,}
	\end{split}
\end{equation} 
thus $0\leq r_1\leq r(t)\leq r_2$ for all $t\geq 0$.
Note that $r(t)$ is defined in such a way that the annulus with outer perimeter equal to the perimeter of $\Omega$ and inner radius $r(t)$ has the same area as $\{x\in\Omega:\,\mathfrak{d}(x)<t\}$.
That is, 
\begin{equation}
	\left|\left\{x\in\R^2:\,r(t)\leq|x|\leq r_2={\textstyle \frac{1}{2\pi}|\partial\Omega|}\right\}\right|
	=\pi\left(\frac{|\partial\Omega|^2}{4\pi^2}-r(t)^2\right)
	=A(t).
\end{equation}

At this point, we want to show the key inequality for this proof:
we will prove that 
\begin{equation}\label{Eq:EstimateR(t)}
	\left|{\textstyle \frac{d}{dt}}r(t)\right| \leq1
\end{equation}
for almost all $t>0$ ---note that, actually, $\frac{d}{dt}r(t)=0$ for all $t>R$ by the second equation in~\eqref{Eq:parallel_endpoints}, and thus we only need to prove it for $t\in (0,R)$.
To establish \eqref{Eq:EstimateR(t)}, we will use the following inequality proved by Sz.-Nagy in \cite{Nagy1959} for simply connected domains:\footnote{
	The following is the rough idea behind \eqref{Eq:NagyIneq}.
	Let $\Gamma$ be a smooth convex closed curve of length $L$, parametrized by arc length $s$, and with curvature $\kappa(s)$. When one moves to the interior parallel curve at distance $t\geq0$, since locally each small arc gets stretched by a factor $1-t\kappa(s)$, the new length is $L(t)=\int(1-t\kappa(s))\,ds=L-2\pi t$, by Gauss-Bonnet theorem in the plane. That is, 
	$L(t)+2\pi t=L(0)$ for convex closed curves.
	Now, if the curve $\Gamma$ is not convex, then inward parallel curves may eventually develop self-intersections (or parts of the curve may disappear) and, thus, the length $L(t)$ may drastically decrease. This leads to $L(t)+2\pi t\leq L(0)$.
	Note also that this shows why we need to assume that $\Omega$ is simply connected: for general domains the integral of the curvature along the boundary is equal to $2\pi$ times the Euler characteristic (which is only $1$ for simply connected domains) ---actually, for an annulus one can readily verify that the contributions from the exterior and interior boundary compensate and thus $L(t)$ is constant (in the appropriate range of $t>0$), hence~\eqref{Eq:NagyIneq} cannot hold in general for non simply connected domains.
}
\begin{equation}\label{Eq:NagyIneq}
	L(t)+2\pi t\leq L(0)=|\partial\Omega|\quad\text{for almost all $t\in (0,R)$.}
\end{equation}
From this and the fact that $\frac{d}{dt}A(t)=L(t)$ for almost all $t>0$, we get
\begin{equation}\label{Eq:parallel2}
	A(t)=\int_0^t L(s)\,ds\leq|\partial\Omega|t-\pi t^2
	\quad\text{for almost all $t\in (0,R)$.}
\end{equation}
Then, using \eqref{Eq:NagyIneq} again and \eqref{Eq:parallel2}, we deduce that for almost all $t\in (0,R)$ it holds
\begin{equation}\label{Eq:parallel3}
	\left({\textstyle \frac{d}{dt}}A(t)\right)^2=L(t)^2\leq(|\partial\Omega|-2\pi t)^2
	=|\partial\Omega|^2-4\pi(|\partial\Omega|t-\pi t^2)
	\leq|\partial\Omega|^2-4\pi A(t).
\end{equation}
Hence, by \eqref{Eq:DefinitionR(t)} we have
\begin{equation}
	\left|{\textstyle \frac{d}{dt}}A(t)\right| \leq 2\pi r(t) 
	\quad\text{for almost all $t\in (0,R)$,}
\end{equation}
and since
\begin{equation}
	\label{Eq:DerivativeR(t)}
	{\textstyle \frac{d}{dt}}r(t)
	=-\frac{\frac{d}{dt}A(t)}{\sqrt{|\partial\Omega|^2-4\pi A(t)}}
	=-\frac{\frac{d}{dt}A(t)}{2\pi r(t)},
\end{equation}
the key inequality \eqref{Eq:EstimateR(t)} follows.

With all these ingredients at hand, we are ready to prove the proposition. Recall that
\begin{equation}\label{Eq:parallel5}
	\mu_\Omega^\mathrm{Rob}(a) = \inf_{u\in H^1(\Omega;\R)\setminus\{0\}}\dfrac{\int_\Omega |\nabla u|^2 + a\int_{\partial\Omega} |u|^2}{\int_\Omega |u|^2}.
\end{equation}
Take in \eqref{Eq:parallel5} any $u\in H^1(\Omega;\R)\setminus\{0\}$ of the form 
\begin{equation}
	u=v(r(\mathfrak{d})) \quad \text{ with } v\in C^1([r_1,r_2]))\setminus\{0\};
\end{equation}
recall that $r_1\leq r(t)\leq r_2$ for all $t\geq 0$. Then, since $|\nabla \mathfrak{d}|=1$ almost everywhere in $\Omega$, using~\eqref{Eq:EstimateR(t)} and the first equation in \eqref{Eq:parallel_endpoints} we see that
\begin{equation}\label{Eq:parallel6}
	\begin{split}
		&|\nabla u|^2=|v'(r(\mathfrak{d}))r'(\mathfrak{d})\nabla \mathfrak{d}|^2
		\leq|v'(r(\mathfrak{d}))|^2|\nabla \mathfrak{d}|,\\
		&u(x)=v(r(\mathfrak{d}(x))=v(r(0))=v(r_2)
		\quad\text{for all }x\in\partial\Omega,\\
		&|u|^2=|v(r(\mathfrak{d}))|^2|\nabla \mathfrak{d}|.
	\end{split}
\end{equation}
Therefore, using \eqref{Eq:parallel6} in \eqref{Eq:parallel5}, the coarea formula, that
\begin{equation}
	2 \pi r(t){\textstyle \frac{d}{dt}}r(t)= -L(t)
\end{equation}
---which follows from \eqref{Eq:DerivativeR(t)}---, and the change of variables $r=r(t)$ ---recall also \eqref{Eq:parallel_endpoints}---, we finally get
\begin{equation} 
	\begin{split}
		\mu_\Omega^\mathrm{Rob}(a) 
		&\leq \dfrac{\int_\Omega |v'(r(\mathfrak{d}))|^2|\nabla \mathfrak{d}| \, d x
			+ a|v(r_2)|^2|\partial\Omega|}{\int_\Omega |v(r(\mathfrak{d}))|^2|\nabla \mathfrak{d}| \, d x}
		=\frac{\int_0^{R}|v'(r(t))|^2 L(t)\,dt
			+ a|v(r_2)|^2|\partial\Omega|}
		{\int_0^{R}|v(r(t))|^2 L(t)\,dt}\\
		&=\frac{-2\pi\int_0^{R}|v'(r(t))|^2 r(t)\frac{d}{dt}r(t)\,dt
			+ a|v(r_2)|^2|\partial\Omega|}
		{-2\pi\int_0^{R}|v(r(t))|^2 r(t)\frac{d}{dt}r(t)\,dt}
		=\frac{\int^{r_2}_{r_1}|v'(r)|^2 \,rdr
			+  a r_2|v(r_2)|^2}
		{\int^{r_2}_{r_1}|v(r)|^2 \,rdr}
	\end{split}
\end{equation}
for all $v\in C^1([r_1,r_2]))\setminus\{0\}$, as desired.
\end{proof}

\section{A PW-type inequality for the $\overline\partial$-Robin Laplacian} \label{Sec:PW_dbarRobin}

In this section we consider the eigenvalue problem for the $\overline\partial$-Robin Laplacian
\begin{equation} 
	\begin{cases}
		-\Delta u = \mu u & \text{in } \Omega, \\
		2\bar \nu \partial_{\bar z}u + au=0 &\text{on } \partial \Omega,
	\end{cases}
\end{equation}
with boundary parameter $a>0$.
Here, $\partial_{\bar z} u = \frac{1}{2}(\partial_1 u+i \partial_2 u)$ and $\overline \nu = \nu_1-i\nu_2$, where $(\nu_1,\nu_2)$ is the outward unit normal vector on $\partial\Omega$.
As mentioned in the introduction, the second step towards the proof of our main result (\Cref{Thm:DPW_QD}) is to obtain a PW-type inequality for the first (smallest) eigenvalue of the $\overline\partial$-Robin Laplacian $\mu$, denoted from now on by $\mu_\Omega(a)$. 

Let us briefly recall (see \cite{Duran2026}) that the operator associated to the above eigenvalue problem, denoted by $\RR_a$, is the self-adjoint operator in $L^2(\Omega;\C)$ defined by
\begin{equation} \label{Eq:RodzinLaplacian}
	\begin{split}
		\Dom(\RR_a) & := \big\{u\in H^1(\Omega;\C): \, \partial_{\bar z} u \in H^1(\Omega;\C), \, 2\bar \nu \partial_{\bar z}u + au = 0 \text{ in } H^{1/2}(\partial \Omega;\C) \big\}, \\
		\RR_a u & := -\Delta u \quad \text{for all } u \in \Dom(\RR_a).
	\end{split}
\end{equation}
In \cite{Duran2026} it is proven that, for $a>0$, the spectrum of 
$\RR_a$ is purely discrete and included in~$(0,+\infty)$. Moreover, the first (smallest) eigenvalue of $\RR_a$ is given by
\begin{equation} \label{Eq:RQ_Rodzin_mu}
	\mu_\Omega(a) = \inf_{u\in E(\Omega)\setminus\{0\}}\dfrac{4\int_\Omega |\partial_{\bar z} u|^2 + a\int_{\partial\Omega} |u|^2}{\int_\Omega |u|^2},
\end{equation}
where $E(\Omega):=\{u\in L^2(\Omega;\C):\, \partial_{\bar z} u\in L^2(\Omega;\C)\text{ and }u\in L^2(\partial\Omega;\C)\}$; see \cite[Theorem 1.2]{Duran2026}.

We next state and prove a PW-type inequality for $\RR_a$.
As will be seen, it is a direct consequence of the analogous result for the (classical) Robin Laplacian, namely \Cref{Thm:PW_Robin}.

\begin{corollary} \label{Cor:PW_dabrRobin}
    Let $\Omega\subset \R^2$ be a bounded simply connected domain with $C^2$~boundary and let~$\delta_\Omega$ denote its isoperimetric deficit, as in \eqref{Eq:IsopDeficit}.
    Then, for all $a>0$ there holds 
    \begin{equation} \label{Eq:PW_dbarRobin_a}
       |\Omega| \, \mu_\Omega \bigg(\frac{a}{\sqrt{|\Omega|}}  \bigg)
       - \pi \, \mu_\disk \bigg(\frac{a}{\sqrt{\pi}}  \bigg) 
       \leq 
       F \left(\sqrt{\mu_\disk \big(\tfrac{a}{\sqrt{\pi}}\big)} \right) \delta_\Omega( \delta_\Omega+2) +  G \left(\sqrt{\mu_\disk \big(\tfrac{a}{\sqrt{\pi}}\big)} \right) \delta_\Omega,
    \end{equation}
    where $F$ and $G$ are as in \eqref{Eq:def_FG}. In particular,
    \begin{equation} \label{Eq:PW_dbarRobin_uniform}
        |\Omega| \, \mu_\Omega \bigg(\frac{a}{\sqrt{|\Omega|}}  \bigg)
        - \pi \, \mu_\disk \bigg(\frac{a}{\sqrt{\pi}}  \bigg) 
        \leq C_{\mathrm{PW}} \delta_\Omega( \delta_\Omega+2) + \pi j_{0,1} \delta_\Omega
    \end{equation}
    and the equality holds if and only if $\Omega$ is a disk. Here, $C_{\mathrm{PW}}>0$ is the constant in \eqref{Eq:PW_Dirichlet}.
\end{corollary}

\begin{proof}
Since $H^1(\Omega;\R) \subset E(\Omega)$ and $4|\partial_{\bar z} u|^2 = |\nabla u|^2$ for every real-valued function $u\in H^1(\Omega;\R)$, in view of the variational characterizations of $\mu_\Omega(a)$ and $\mu_\Omega^\mathrm{Rob}(a)$ given in  \eqref{Eq:RQ_Rodzin_mu} and \eqref{Eq:RQ_Robin_mu}, respectively, we trivially have 
\begin{equation} \label{Eq:RodzinLEQRobin}
    \mu_\Omega(a) \leq \mu_\Omega^\mathrm{Rob}(a) \quad \text{for all } a>0.
\end{equation}
Moreover, by the explicit computations in \cite[Appendix A]{Duran2026} ---in particular, see \cite[Corollary~A.3]{Duran2026}---, equality holds in \eqref{Eq:RodzinLEQRobin} when $\Omega=\disk$. Therefore, the corollary is an immediate consequence of \Cref{Thm:PW_Robin}.
\end{proof}

\section{Proof of \Cref{Thm:DPW_QD}} \label{Sec:Proof_MainThm}

In this section we prove \Cref{Thm:DPW_QD}, our main result.
To do it, we use \Cref{Cor:PW_dabrRobin} and a connection between $\lambda_\Omega(\theta,m)$ and $\mu_\Omega(a)$ studied in \cite{DuranMasSanzPerela2026} and refined in \cite{Duran2027}.
More precisely, we will use \cite[Theorem~2.2~$(i)$ and~$(iii)$]{Duran2027}, a result that allows to transfer upper bounds from~$\mu_\Omega(a)$ to~$\lambda_\Omega(\theta,m)$ and vice versa.
For the benefit of the reader, we recall its statement below. 
In there, $\vartheta$ is the smooth, strictly decreasing, and bijective function 
\begin{equation} \label{Eq:vartheta_function}
   \begin{split}
        {\textstyle \vartheta: (-\frac \pi 2, \frac \pi 2)} & \rightarrow  (0,+\infty)\\
             \theta  & \mapsto  {\textstyle\frac{1-\sin\theta}{\cos\theta}}.
    \end{split}
\end{equation}

\begin{lemma}[\cite{Duran2027}] \label{Prop:recipe_bounds}
    Let $\Omega\subset \R^2$ be a bounded domain with $C^2$ boundary. Given $m\geq 0$ and~$\mathcal B>0$, let~$a>0$ and~$\theta\in(-\frac \pi 2, \frac \pi 2)$ be such that
        $a = \vartheta(\theta) ( m + \mathcal B )$.
    Then, 
    \begin{equation}
    	\lambda_\Omega(\theta,m) \leq \mathcal B\quad\text{if and only if}\quad \mu_\Omega(a) \leq \mathcal B^2-m^2.
    \end{equation}
    Moreover, $\lambda_\Omega(\theta,m) = \mathcal B $ if and only if $ \mu_\Omega(a) = \mathcal B^2-m^2$.
\end{lemma}

The strategy to prove \Cref{Thm:DPW_QD} will be to use the previous lemma with a suitable $\mathcal B$ and take advantage of \Cref{Cor:PW_dabrRobin} ---in particular, the  inequality with right-hand side independent of~$a$--- and the concavity of the function $a\mapsto \mu_\disk(a)$.
Since we will use this and other properties of this function (some of them established in \cite{Duran2026, Duran2027}), we collect them in the following lemma, whose proof is postponed until the end of the section.

\begin{lemma}\label{Lemma:PropertiesMuDisc}
	Given $a>0$, let $\mu_\disk(a)$ denote the first eigenvalue of the $\overline{\partial}$-Robin Laplacian $\RR_{a}$~in the unit disk and, given $\theta\in(-\frac \pi 2, \frac \pi 2)$ and $m\geq 0$, let $\lambda_\disk$ denote the first positive eigenvalue of the quantum dot Dirac operator $\mathcal{D}_{\theta, m}$ in the unit disk; that is, $\lambda_\disk := \lambda_\disk(\theta,m)$.~Then, 
	\begin{enumerate}[label=$(\roman*)$]
		\item the function $a\mapsto \mu_\disk(a)$ is real analytic and concave in $(0,+\infty)$;
		
		\item it holds that
		\begin{equation} \label{Eq:QD_dbar_disk}
			\mu_\disk \big( \vartheta(\theta) (\lambda_\disk+m) \big) = \lambda_\disk^2-m^2\text{;}
		\end{equation}
		\item the derivative of $a\mapsto \mu_\disk(a)$  satisfies
		\begin{equation}\label{Eq:DerivativeMuDiskIneq}
			a\frac{d}{da}\mu_\disk(a) \leq \mu_\disk(a) \quad \text{ for all } a>0
		\end{equation}
		and
		 \begin{equation}\label{Eq:DerivativeMuDiskId}
			\frac{d}{da}\mu_{\disk}(a)\Big|_{a=\vartheta(\theta)(\lambda_\disk+m)} = \dfrac{2}{1+ \frac{\lambda_\disk+m}{\lambda_\disk-m}\vartheta(\theta)^2 }.
		\end{equation}
	\end{enumerate}
\end{lemma}

With the previous ingredients at hand we proceed now with the proof of our main result.

\begin{proof}[Proof of \Cref{Thm:DPW_QD}]

As we did in the proof of~\Cref{Thm:PW_Robin} for the first eigenvalue of the Robin Laplacian, the first step will be to reduce the proof of \eqref{Ineq_:main_thm} to domains with~$|\Omega|=\pi$ in virtue of the scaling property 
\begin{equation} \label{Eq:scaling_Dirac}
    \lambda_{t\Omega} (\theta,m) = \dfrac{1}{t} \lambda_\Omega(\theta,tm) \quad \text{for all } t>0,
\end{equation}
which easily follows by inspecting  \eqref{Eq:Dirac_op_theta}.
Let $\Omega\subset\R^2$ be a bounded simply connected domain with $C^2$~boundary, and let $\delta_{\Omega}$ denote~its isoperimetric deficit; recall \eqref{Eq:IsopDeficit}. Assume that for all~$\theta\in(-\frac \pi 2, \frac \pi 2)$ and all~$\widetilde{m}\geq 0$ we know that
\begin{equation}
	\label{Eq:PW_DiracVolumePi}
	\lambda_{\widetilde\Omega} \left (\theta, \widetilde{m} \right) - \lambda_\disk(\theta,\widetilde{m}) \leq C({\theta,\widetilde{m}}) \, \delta_{\widetilde\Omega}
\end{equation}
holds for every bounded simply connected domain $\widetilde\Omega\subset \R^2$ with $C^2$~boundary such that $|\widetilde\Omega|=\pi$, for some constant $C({\theta,\widetilde{m}})$ depending only on $\theta$ and $\widetilde{m}$.
Then, by the scale invariance of the isoperimetric deficit and the scaling property \eqref{Eq:scaling_Dirac}, taking  in \eqref{Eq:PW_DiracVolumePi}
\begin{equation}
	\widetilde{m} := \frac{m}{\sqrt{\pi}} 
	\quad \text{ and } \quad \widetilde \Omega := \sqrt{\frac{\pi}{|\Omega|}} \Omega \quad\text{---which satisfies $|\widetilde \Omega|=\pi$ and 
	$\delta_{\widetilde \Omega}=\delta_\Omega$---,}
\end{equation}
we deduce that
\begin{equation}
	\label{Eq:ConclusionScalingPWDirac}
	\begin{split}
		 \sqrt{|\Omega|} \, \lambda_\Omega \left(\theta, \frac{m}{\sqrt{|\Omega|}}  \right)
		- \sqrt{\pi} \lambda_\disk \left(\theta, \frac{m}{\sqrt{\pi}}  \right) 
		&= \sqrt{\pi } \left( \lambda_{\widetilde\Omega}(\theta,\widetilde{m}) - \lambda_\disk(\theta,\widetilde{m}) \right) \\
		& \leq \sqrt{\pi } C({\theta,\widetilde{m}}) \, \delta_{\widetilde\Omega} = \sqrt{\pi } C({\theta,\tfrac{m}{\sqrt{\pi}} }) \, \delta_{\Omega},
	\end{split}       
    \end{equation}
which is \eqref{Ineq_:main_thm} with $C= \sqrt{\pi } C({\theta,{m}/{\sqrt{\pi}} })$. 
This shows that, to prove the theorem, it is enough to establish
\begin{equation} \label{eq:AUXgoal}
    \lambda_\Omega(\theta,m) - \lambda_\disk(\theta,m) \leq C(\theta,m) \,  \delta_\Omega
\end{equation}
under the assumption $|\Omega|=\pi$. To shorten the notation, in what follows we denote $\lambda_\Omega:=\lambda_\Omega(\theta,m)$,  $\lambda_\disk:=\lambda_\disk(\theta,m)$, and~$\delta:=\delta_\Omega$.

To establish \eqref{eq:AUXgoal}, let $C_\star>0$ be a constant to be chosen later, and set
\begin{equation}
	a_0:= \vartheta(\theta) (m + \lambda_\disk+ C_\star\delta).
\end{equation}
Note that since $ \lambda_\disk > m\geq 0$ by \cite[Lemma 3.1]{DuranMasSanzPerela2026}, $\theta\in (-\frac{\pi}{2},\frac{\pi}{2})$, and $C_\star \delta \geq 0$, we have that~$a_0\in(0,+\infty)$. With this choice of the parameter $a_0$ and using \Cref{Prop:recipe_bounds} for $\mathcal{B}:=\lambda_\disk+ C_\star \delta$, we see that \eqref{eq:AUXgoal} holds for $C({\theta, m}) = C_\star$ if and only if 
\begin{equation}\label{Ineq_:main_thm_3}
	\mu_\Omega(a_0) - \left( \left(\lambda_\disk+ C_\star \delta\right)^2-m^2 \right) \leq 0.
\end{equation}
Our goal now is to prove \eqref{Ineq_:main_thm_3}. 
To this end, using that $a_0= \vartheta(\theta) (\lambda_\disk+ m +  C_\star\delta)$ we can rewrite the left hand side of \eqref{Ineq_:main_thm_3} as
\begin{equation}\label{Ineq_:main_thm_4}
	\begin{split}
		\mu_\Omega&(a_0) - \left( \left(\lambda_\disk+ C_\star \delta\right)^2-m^2 \right) \\ &= \Big(\mu_\Omega(a_0) - \mu_\disk(a_0)\Big) + \mu_\disk \Big( \vartheta(\theta) (\lambda_\disk+m) + \vartheta(\theta) C_\star \delta \Big)  - \Big((\lambda_\disk + C_\star \delta)^2 -m^2\Big)\\
		&=: M_1+M_2+M_3. 
	\end{split}
\end{equation}
We estimate $M_1$ using the PW-type inequality for the $\overline{\partial}$-Robin Laplacian $\RR_{a_0}$; recall~\Cref{Cor:PW_dabrRobin}. In particular (and this is crucial), we will use the inequality with the right-hand side independent of the boundary parameter $a=a_0$.
We get
\begin{equation}
    \label{Eq:ProofMainResultPWRodzin}
	\begin{split}
		M_1=\mu_\Omega(a_0) - \mu_\disk(a_0)
		\leq \frac{C_{\mathrm{PW}}}{\pi} \, \delta( \delta+2) + j_{0,1}  \delta,
	\end{split}
\end{equation}
since $|\Omega|=\pi$. To estimate $M_2$ we use the properties of $a\mapsto \mu_\disk(a)$ stated in~\Cref{Lemma:PropertiesMuDisc}. More specifically, we first use that this function is differentiable and concave, hence $\mu_\disk(b)\leq\mu_\disk(a)+\frac{d}{da}\mu_{\disk}(a)(b-a)$ for all $b\geq a>0$, and then we use~\eqref{Eq:QD_dbar_disk} and~\eqref{Eq:DerivativeMuDiskId} to compute $\mu_\disk(a)$ and its derivative at the specific value $a=\vartheta(\theta)(\lambda_\disk+m)$. We get
\begin{equation}
	\begin{split}
		M_2&=\mu_\disk \Big( \vartheta(\theta) (\lambda_\disk+m) + \vartheta(\theta) C_\star\delta \Big)\\
		&\leq\mu_\disk \Big( \vartheta(\theta) (\lambda_\disk+m)\Big) 
		+ \frac{d}{da}\mu_{\disk}(a)\Big|_{a=\vartheta(\theta)(\lambda_\disk+m)}\vartheta(\theta) C_\star \delta
		= \lambda_\disk^2-m^2 
		+ \dfrac{2\vartheta(\theta) C_\star \delta}{1+ \frac{\lambda_\disk+m}{\lambda_\disk-m}\vartheta(\theta)^2 }.
	\end{split}
\end{equation}
Now, using these estimates for $M_1$ and $M_2$ in \eqref{Ineq_:main_thm_4}, we get    
\begin{equation}
	\begin{split}
		\mu_\Omega&(a_0)  - \left( \left(\lambda_\disk+ C_\star \delta\right)^2-m^2 \right)  
		=  M_1+M_2+M_3\\
		& \leq \frac{C_{\mathrm{PW}}}{\pi} \delta( \delta+2) + j_{0,1} \delta
		+\lambda_\disk^2-m^2 
		+ \dfrac{2\vartheta(\theta) C_\star\delta}{1+ \frac{\lambda_\disk+m}{\lambda_\disk-m}\vartheta(\theta)^2 }
		- \left(\lambda_\disk^2-m^2+2\lambda_\disk C_\star\delta + C_\star^2\delta^2\right)\\
		&=\left(\frac{2 C_{\mathrm{PW}}}{\pi}  + j_{0,1} 
		+ 2\left(\dfrac{\vartheta(\theta) }{1+ \frac{\lambda_\disk+m}{\lambda_\disk-m}\vartheta(\theta)^2 }
		-\lambda_\disk  \right)C_\star\right)\delta 
		+\left(\frac{C_{\mathrm{PW}}}{\pi} -C_\star^2 \right)\delta^2 \\
		&=: \ell \delta + q \delta ^2.
	\end{split}
\end{equation}
Since $\delta\geq0$, in order to establish \eqref{Ineq_:main_thm_3} and conclude the proof, we need to show that both coefficients $\ell$ and $q$ are nonnegative, which will hold if we take $C_\star$ big enough, as shown next. On the one hand, $q\leq 0$ provided that
\begin{equation}
	\label{Eq:ConditionCStar1}
	C_\star \geq \sqrt{\frac{C_{\mathrm{PW}}}{\pi}}.
\end{equation}
On the other hand, to see that $\ell\leq 0$ for $C_\star$ big enough we only need to show that 
\begin{equation}\label{Eq:AuxCoef>0}
	\dfrac{\vartheta(\theta)}{1+ \frac{\lambda_\disk+m}{\lambda_\disk -m} 
		\vartheta(\theta)^2 }<\lambda_\disk.
\end{equation}
To see this we use \eqref{Eq:DerivativeMuDiskId}, \eqref{Eq:DerivativeMuDiskIneq}, and \eqref{Eq:QD_dbar_disk} to obtain
\begin{equation}
	\begin{split}
		\dfrac{2\vartheta(\theta)}{1+ \frac{\lambda_\disk+m}{\lambda_\disk-m}\vartheta(\theta)^2 }&=\vartheta(\theta)\frac{d}{da}\mu_{\disk}(a)\Big|_{a=\vartheta(\theta)(\lambda_\disk+m)} 
		= \frac{1}{\lambda_\disk+m} \vartheta(\theta)(\lambda_\disk+m) \frac{d}{da}\mu_{\disk}(a)\Big|_{a=\vartheta(\theta)(\lambda_\disk+m)} \\
		& \leq \dfrac{1}{\lambda_\disk+m} \mu_{\disk}\big(\vartheta(\theta)(\lambda_\disk+m)\big) 
		= \dfrac{1}{\lambda_\disk+m} ( \lambda_\disk^2-m^2 ) 
		= \lambda_\disk - m 
		< 2\lambda_\disk,
	\end{split}
\end{equation}
where in the last inequality we  used that~$\lambda_\disk>m\geq 0$ by \cite[Lemma 3.1]{DuranMasSanzPerela2026}.
This shows \eqref{Eq:AuxCoef>0} and, as a consequence, we deduce that $\ell \leq 0$ provided that
\begin{equation}
	\label{Eq:ConditionCStar2}
	C_\star \geq \dfrac{\frac{ C_{\mathrm{PW}}}{\pi}  + \frac{j_{0,1}}{2} }{
		\lambda_\disk - \tfrac{\vartheta(\theta) }{1+ \tfrac{\lambda_\disk+m}{\lambda_\disk-m}\vartheta(\theta)^2 }}
\end{equation}
In summary, if $C_\star$ satisfies \eqref{Eq:ConditionCStar1} and \eqref{Eq:ConditionCStar2}, then \eqref{eq:AUXgoal} holds for any $C({\theta, m}) \geq C_\star$.

Let us conclude this proof by analyzing the case of equality mentioned in~\Cref{Rmk:Constant}.
Again by scaling it suffices to consider \eqref{eq:AUXgoal} in the case $|\Omega|= \pi$.
On the one hand, if $\Omega = \disk$ then clearly \eqref{eq:AUXgoal} holds with an equality. 
On the other hand, if $\Omega$ is not a disk, then the inequality in \eqref{Eq:ProofMainResultPWRodzin} is strict; see the statement of the PW-type inequality for the $\overline\partial$-Robin Laplacian in~\Cref{Cor:PW_dabrRobin}.
As a consequence, we get a strict inequality in \eqref{Ineq_:main_thm_3} for the $C_\star$ chosen above, which gives a strict inequality in \eqref{eq:AUXgoal} by~\Cref{Prop:recipe_bounds}.
\end{proof}

It only remains to prove \Cref{Lemma:PropertiesMuDisc}.
As mentioned, some of the stated properties were already proven in~\cite{Duran2026,Duran2027}.

\begin{proof}[Proof of \Cref{Lemma:PropertiesMuDisc}]
	We start with the properties already proved in (or readily following from) \cite{Duran2026,Duran2027}.
	First, the fact that the function $(0,+\infty)\ni a\mapsto	\mu_\disk(a)\in(0,+\infty)$ is everywhere analytic follows from~\cite[Theorem~1.3 and~Remark~4.3]{Duran2026} ---which gives that $a\mapsto	\mu_\disk(a)$ is  analytic except at values $a$ in which the multiplicity of the eigenvalue $\mu_\disk(a)$ changes--- combined with~\cite[Corollary~A.3]{Duran2026} ---which shows that $\mu_\disk(a)$ is simple;  
the concavity of the function is proven in \cite[Theorem 1.3 $(iv)$]{Duran2026}.
	Second, the identity \eqref{Eq:QD_dbar_disk}, that is, $\mu_\disk \big( \vartheta(\theta) (\lambda_\disk+m) \big) = \lambda_\disk^2-m^2$, is precisely the case of equality in \Cref{Prop:recipe_bounds} with $\mathcal{B} = \lambda_\disk$.
	Last, that $a\frac{d}{da}\mu_\disk(a) \leq \mu_\disk(a)$ holds for all $a>0$ follows by the variational characterization \eqref{Eq:RQ_Rodzin_mu} for $\mu_\disk(a)$ using the explicit formula~$\frac{d}{da}\mu_\disk(a) = \int_{\partial\disk} |u|^2 / \int_{\disk} |u|^2 $ given in~\cite[Theorem~1.3~$(iii)$]{Duran2026}; here $u$ is a minimizer in~\eqref{Eq:RQ_Rodzin_mu}.

	Let us now address the proof of \eqref{Eq:DerivativeMuDiskId}, that is,
	\begin{equation}\label{Eq:DerivativeMuDiskIdProof}
		\frac{d}{da}\mu_{\disk}(a)\Big|_{a=\vartheta(\theta)(\lambda_\disk+m)} = \dfrac{2}{1+ \frac{\lambda_\disk+m}{\lambda_\disk-m}\vartheta(\theta)^2 }.
	\end{equation}
	By \cite[Proposition~A.2]{Duran2026} (see also \cite[Corollary A.3]{Duran2026}), the first eigenfunction of $\RR_{\vartheta(\theta)(\lambda_\disk+m)}$ in~$\disk$ is (up to a multiplicative constant) the radial function
    \begin{equation}
        {\textstyle u(r):=J_0 \big(\sqrt{\mu_\disk(\vartheta(\theta)(\lambda_\disk+m))}r \big) = J_0 \big(\sqrt{\lambda_\disk^2-m^2}\, r \big),\quad r\in[0,1)};
    \end{equation}
    this last equality holding by \eqref{Eq:QD_dbar_disk}, already proved in the previous paragraph.
	Therefore, by the explicit formula for $\frac{d}{da}\mu_\disk(a)$ mentioned above (and proved in \cite[Theorem~1.3~$(iii)$]{Duran2026}) we have
	\begin{equation}
		\frac{d}{da}\mu_{\disk}(a)\Big|_{a=\vartheta(\theta)(\lambda_\disk+m)} = \dfrac{\int_{\partial\disk} |u|^2}{\int_{\disk} |u|^2} = \dfrac{2}{1+\frac{J_1^2\big(\sqrt{\lambda_\disk^2-m^2} \big)}{J_0^2\big(\sqrt{\lambda_\disk^2-m^2} \big)}} = \dfrac{2}{1+ \frac{\lambda_\disk+m}{\lambda_\disk-m}\vartheta(\theta)^2 },
	\end{equation}
	where we have used \cite[Equation (11.3.34)]{Abramowitz1964} to compute the integral involving $J_0$ and, in the last equality,  that the first positive eigenvalue $\lambda_\disk$ of $\D_{\theta,m}$ in $\disk$ is the smallest positive solution~$\lambda$ to the eigenvalue equation \eqref{Eq:EigenEq_disk}, which we recall to be
	\begin{equation}
		(\lambda+m)\vartheta(\theta) J_0 \big(\sqrt{\lambda^2-m^2}\big) - \sqrt{\lambda^2-m^2} J_1 \big(\sqrt{\lambda^2-m^2}\big) = 0.
	\end{equation}
This concludes the proof.
\end{proof}

\end{document}